\pdfoutput=1
\documentclass[12pt,twoside,reqno]{amsart}
\usepackage{amsfonts,amssymb,amsmath,amsthm,amsxtra}
\usepackage{mathtools}
\usepackage{bm}
\usepackage{xparse}
\usepackage{mathtools}
\usepackage{dsfont}
\usepackage{bbm}
\usepackage{mathrsfs}
\usepackage{tikz-cd}
\usepackage{enumitem}
\usepackage{comment}
\usepackage{lmodern}
\usepackage[top=1.5in, bottom=1.5in, left=1.23in, right=1.23in]{geometry}
\usepackage[all]{xy}
\usepackage{amsmath}
\usepackage{amssymb}
\usepackage{amsthm}
\usepackage{amscd}

\DeclareMathAlphabet{\mathbfit}{OT1}{cmss}{bx}{it}
\DeclareMathAlphabet\mathbfcal{OMS}{cmsy}{b}{n}

\usepackage{scalerel,stackengine}
\stackMath
\newcommand\reallywidehat[1]{%
\savestack{\tmpbox}{\stretchto{%
  \scaleto{%
    \scalerel*[\widthof{\ensuremath{#1}}]{\kern-.7pt\bigwedge\kern-.7pt}%
    {\rule[-\textheight/2]{1.5ex}{\textheight}}%WIDTH-LIMITED BIG WEDGE
  }{\textheight}% 
}{0.9ex}}%
\stackon[2.25pt]{#1}{\tmpbox}%
}

\DeclareMathAlphabet{\mathbfit}{OT1}{cmss}{bx}{it}
\DeclareMathAlphabet\mathbfcal{OMS}{cmsy}{b}{n}

\newcommand{\n}{\mathcal{N}}

\numberwithin{equation}{section}

\newtheorem{theorem}[equation]{Theorem}

\newtheorem{lemma}[equation]{Lemma}

\theoremstyle{definition}
\newtheorem{remark}[equation]{Remark}

\usepackage{hyperref} %,hypertexnames=false,colorlinks,[pagebackref]
\hypersetup{
    colorlinks=true,       % false: boxed links; true: colored links
    linkcolor=blue,          % color of internal links
    citecolor=magenta,        % color of links to bibliography
    filecolor=magenta,      % color of file links
    urlcolor=cyan           % color of external links
}

\begin{document}
\title[Multiparameter extensions:Strong variational bounds]{Multiparameter extensions of the Christ--Kiselev maximal theorem: strong variational bounds}

\author{Himali Dabhi}
\address[Himali Dabhi]{School of mathematics, the university of edinburgh and the maxwell institute for the mathematical sciences, 
James Clerk Maxwell Building,
The King's Buildings,
Peter Guthrie Tait Road,
City Edinburgh,
EH9 3FD}
\email{h.p.dabhi@sms.ed.ac.uk}

\newcommand{\bk}[1]{{\color{red}{#1}}}

\renewcommand{\thefootnote}{}
\footnote{\emph{To appear in Bulletin of the London Mathematical Society}}
\renewcommand{\thefootnote}{\arabic{footnote}}
\setcounter{footnote}{0}

\setcounter{tocdepth}{1}
\begin{abstract}
    For a linear operator $T$ bounded from $L^p(Y)$ to $L^q(X)$, the Christ--Kiselev theorem gives $L^p \to L^q$ bounds for the maximal function  $T^{*}$ associated to a filtration on $Y$. This result has been extended to the multiparameter setting, establishing bounds for the maximal function associated to a product of filtrations. In this note, we strengthen these results by proving $\rho$-variational bounds for the multiparameter truncations when $\rho>p$. Furthermore, we give a multilinear extension of the Christ--Kiselev theorem, establishing strong variational bounds for the multiparameter truncations for a general multilinear operator.
\end{abstract}

\maketitle

\section{Introduction} In their influential paper \cite{CK}, M. Christ and A. Kiselev establish a general maximal function bound on an arbitrary measure space $(Y,\nu)$ with a filtration
$\{{\mathcal Y}_n\}_{n\in {\mathbb Z}}$ (defined as a family of measurable sets that are nested, ${\mathcal Y}_n \subseteq {\mathcal Y}_{n+1}$ for all $n\in {\mathbb Z}$). Consider another measure space $(X,\mu)$ and suppose $T: L^p(Y) \to L^q(X)$
is a bounded linear operator with the operator norm $\|T\|_{L^{p} \to L^{q}}$. We associate to $T$ the corresponding
maximal function
$$
T^{*} f(x) \ := \ \sup_{n\in {\mathbb Z}} |T(f \mathbbm{1}_{\mathcal Y_n})(x)|,
$$
for which we have the following.

\begin{theorem}\label{CK-thm} (Christ--Kiselev maximal theorem \cite{CK}) If $1 \leq p < q \leq \infty$, then $T^{*}$ is
bounded from $L^p(Y)$ to $L^q(X)$. Moreover,
\begin{equation}\label{CK-pq}
\|T^{*}\|_{L^{p} \to L^{q}} \ \le \ (1 - 2^{-(p^{-1} - q^{-1})})^{-1} \, \|T\|_{L^{p}\to L^{q}}.
\end{equation}
\end{theorem}

Since the bound (\ref{CK-pq}) is independent of the filtration $\{\mathcal{Y}_N\}_{N \in \mathbb Z}$, the same bound holds for continuous filtrations $\{{\mathcal Y}_t\}_{t\in {\mathbb R}}$.The proof of Theorem \ref{CK-thm} given in \cite{CK} is elegant and very robust; it extends to Banach-valued functions and such an extension has found applications to Strichartz estimates (see for example, \cite{Tao}). Furthermore, it gives a unified approach to previous results in the literature. For example, inspired by Menshov's \cite{menshov} almost everywhere pointwise convergence result for orthogonal expansions $\sum_{n=1}^{\infty} a_n \phi_n(x)$ whenever $\mathbf{a} := \{a_n\} \in \ell^p(\mathbb N), 1\le p <2$ and Paley's \cite{P} upgrade to a maximal function bound
\begin{equation}\label{menshov-paley}
\| S^{*}(\mathbf{a})\|_{L^{p'}} \le C_p \,
\bigl(\sum_{n=1}^{\infty} |a_n|^p \bigr)^{1/p} \ \ {\rm for} \ \ 1\le p < 2, \ \ \frac{1}{p}+\frac{1}{p'}=1,
\end{equation}
where $S^{*}(\mathbf{a})(x) := \sup_N \, \bigl| \sum_{n=1}^N a_n \phi_n(x) \bigr|$,
Zygmund \cite{zygmund} proved
\begin{equation}\label{zyg-FT}
\| f^{*}\|_{L^{p'}({\mathbb R})} \le C_p \, \|f\|_{L^p({\mathbb R})} \ \ {\rm for} \ \ 1\le p < 2,
\end{equation}
where
$$
f^{*}(x) \ := \ \sup_{R>0} \ \Bigl| \int_{|y|\le R} f(y) {\rm e}^{-2\pi i x y} dy \Bigr|,
$$
the so-called maximal Hausdorff--Young inequality. Furthermore Menshov (and independently Rademacher) showed that \eqref{menshov-paley} fails when $p=2$ for general orthonormal systems. On the other hand, \eqref{menshov-paley} is true for the trigonometric system when $p=2$. This is the famous Carleson--Hunt theorem (see \cite{C} and \cite{H}) which states that
$$\mathcal{C}f(x) \ := \ \sup_{R>0}\Bigl| \int_{|y| \leq R} \hat{f}(y){\rm e}^{2 \pi i xy}dy\Bigr|$$
is bounded on $L^2(\mathbb R)$ and is equivalent by transference to $S^*$ being bounded from $\ell^2(\mathbb Z)$ to $L^2(\mathbb T)$ for the classical trigonometric system $\{\phi_n(x)=e^{2 \pi i n x}\}_{n \in \mathbb Z}$.

The generality of Theorem \ref{CK-thm} allows us to establish higher dimensional variants of the maximal Hausdorff--Young inequality. For a sequence $\mathbf{a}: = \{a_{{\underline{n}}}\}_{{\underline{n}}\in {\mathbb Z}^d} \in \ell^p({\mathbb Z}^d)$ and a function $f \in L^p({\mathbb R}^d)$, define
$$
S^{*}(\mathbf{a})({\underline{\xi}}) \ := \ \sup_{R>0} \ \Bigl|\sum_{\|{\underline{n}}\| \le R} a_{{\underline{n}}} {\rm e}^{- 2\pi i {{\underline{n}}\cdot {\underline{\xi}}}} \Bigr|
\ \ \ {\rm and} \ \ \ f^{*}({\underline{x}}) \ := \  \sup_{R>0}\ \Bigl| \int_{\|{\underline{y}}\|\le R} f({\underline{y}}) {\rm e}^{-2\pi i {\underline{x}} \cdot {\underline{y}}}\, d{\underline{y}} \Bigr|,
$$
where ${\underline{\xi}} \in {\mathbb T}^d$, ${\underline{x}} \in {\mathbb R}^d$ and $\|\cdot\|=\|\cdot\|_{\ell^2_d}$ is the standard Euclidean norm on $\mathbb R^d$. 
Then Theorem \ref{CK-thm} implies the dimension-free bounds
\begin{equation}\label{2bounds}
\|S^{*}(\mathbf{a}) \|_{L^{p'}({\mathbb T}^d)} \le C_{p} \, \|\mathbf{a}\|_{\ell^p({\mathbb Z}^d)}
\ \ \ {\rm and} \ \ \
\|f^{*} \|_{L^{p'}({\mathbb R}^d)} \le C_{p} \, \|f\|_{L^p({\mathbb R}^d)}
\end{equation}
whenever $1 \le p < 2$, establishing almost everywhere pointwise convergence for the spherical partial sums and integrals
$$
S_R(\mathbf{a})({\underline{\xi}}) :=  \sum_{\|{\underline{n}}\| \le R} a_{{\underline{n}}} {\rm e}^{- 2\pi i {{\underline{n}}\cdot {\underline{\xi}}}} \ \ \ {\rm and} \ \ \
f_R({\underline{x}}) := \int_{\|{\underline{y}}\|\le R} f({\underline{y}}) {\rm e}^{-2\pi i {\underline{x}} \cdot {\underline{y}}} d{\underline{y}}
$$
whenever $\mathbf{a} \in \ell^p({\mathbb Z}^d)$, $f \in L^p({\mathbb R}^d)$ and $1\le p < 2$. We could also replace the Euclidean norm $\ell^2_d$ by $\ell^q_d$, $1 \leq q \leq \infty$ and still have the same result. At the endpoint $p=2$, the maximal function bounds in (\ref{2bounds}) hold if we consider truncation with respect to the $\ell^1_{d}$ and $\ell^{\infty}_{d}$ norm; this is an extension of the Carleson--Hunt theorem in higher dimensions proved by C. Fefferman \cite{F1}. It is a major open problem to determine whether or not the result holds if we consider $\ell^{q}_{d}$ norms for $1<q<\infty$.

In $2004$, Y. Cho, E. Koh and S. Lee extended Theorem \ref{CK-thm} for Lorentz spaces. We recall the definition of Lorentz spaces. For a complex-valued measurable function $f$ on a measure space $(Y, \nu)$, $0 <p<\infty$ and $0 < r \leq \infty$, we define
$$\|f\|_{L^{p,r}(Y)}:=\begin{cases}
    \bigl(\displaystyle\int [t^{1/p} f^{*}(t)]^{r} \frac{dt}{t}\bigr)^{1/r}, & \textit{if $r < \infty$}, \\
    \sup_{t>0}t^{1/p}f^{*}(t), & \textit{if $r=\infty$},
\end{cases}$$
where $f^{*}$ is the decreasing rearrangement of $f$. We define $L^{p,r}(Y)$ to be the space of all measurable functions $f$ such that $\|f\|_{L^{p,r}(Y)} < \infty$.

\begin{theorem}\label{CKL}(Christ--Kiselev theorem for Lorentz spaces \cite{CKL})
Let $1 \leq p \leq r < s \leq q < \infty$, and suppose that $T:L^{p,r}(Y) \to L^{q,s}(X)$ is a bounded linear operator. Then $T^{*}$ is bounded from $L^{p,r}(Y)$ to $L^{q,s}(X)$. Moreover, 
\begin{equation}\label{CKLpq}
    \|T^{*}\|_{L^{p,r} \to L^{q,s}} \le (1-2^{-(r^{-1}-s^{-1})})^{-1}\|T\|_{L^{p,r} \to L^{q,s}}.
\end{equation}
\end{theorem}

We highlight that Theorem \ref{CKL} is an extension of the Christ--Kiselev theorem illustrating the robustness of the original Christ--Kiselev argument.

So far, we have considered one-parameter truncations, but in higher dimensions we can also consider multiparameter truncations. For a sequence $\mathbf{a}=\{a_{m,n}\}_{(m,n) \in \mathbb Z^2}$ and $f \in L^p(\mathbb R^2)$, we define the two-parameter truncations given by
\begin{equation}\label{mps}
 S_{M,N}(\mathbf{a})(\theta,\omega)\:=\ \sum_{|m|\leq M \atop |n| \leq N}a_{m,n}{\rm e}^{-2\pi i (m\theta+n\omega)} \ ; \ f_{M,N}(\xi,\eta)\:=\  \int_{|x| \leq M \atop |y| \leq N}f(x,y){\rm e}^{-2 \pi i(x\xi+y \eta)}dx dy,
\end{equation}
where $(\theta, \omega) \in \mathbb T^2$ and $(\xi, \eta) \in \mathbb R^2$.
It was shown by Bulj and Kova\v c in \cite{BK} that the maximal operators
\begin{equation*}
S^{*}(\mathbf{a})(\theta,\omega) \ := \ \sup_{M,N>0}|S_{M,N}(\mathbf{a})(\theta,\omega)| \ \ \ {\rm and} \ \ \ f^{*}(\xi,\eta) \ := \ \sup_{M,N>0}|f_{M,N}(\xi,\eta)|
\end{equation*}
satisfy the bounds
$$\|S^{*}(\mathbf{a}) \|_{L^{p'}({\mathbb T}^2)} \le C_{p} \, \|\mathbf{a}\|_{\ell^p({\mathbb Z}^2)}
\ \ \ {\rm and} \ \ \
\|f^{*} \|_{L^{p'}({\mathbb R}^2)} \le C_{p} \, \|f\|_{L^p({\mathbb R}^2)}$$
for $1 \leq p <2.$

In this paper we prove the Christ--Kiselev theorem, for multiparameter truncations with respect to a general multilinear operator. Furthermore we will show that the maximal bounds (\ref{CK-pq}) and $(\ref{CKLpq})$ can be improved to strong variational bounds. This is one of the few instances where variational bounds are established in the multiparameter setting; see \cite{MSW} for more context on the difficulty of establishing multiparameter variational bounds. 

Establishing variational bounds is important in ergodic theory because maximal function bounds are often not sufficient to imply pointwise almost everywhere limits since there are no natural dense classes of functions for which the studied pointwise convergence is easy to verify in the setting of abstract measure spaces. Due to recent connections with ergodic theory, variational bounds in harmonic analysis have become a central topic over the last 15 to 20 years, see \cite{JSW} for instance. Also in \cite{osttw}, it was observed that an upgrade of the maximal Hausdorff--Young inequality to a variational bound can be used to give another proof of the Christ--Kiselev result on generalised eigenfunctions of 1-d Schr\"odinger equations, bypassing the multilinear analysis in \cite{CK}.

\subsection{$\rho$-variation semi-norm} Let $\mathbb I \subseteq \mathbb R^{d}$ be given with the following partial order relation ``$\preceq$''; if $(s_1,\ldots,s_d) \preceq (t_1,\ldots,t_d)$ then $s_i \leq t_i$ for all $i=1,2,\ldots,d$. Suppose that $\{F_{\underline{t}}:\underline{t} \in \mathbb I\}$ is a family of complex-valued measurable functions defined on $X$. For $1 \leq \rho < \infty$, we define the $\rho$-variation of the family $\{F_{\underline{t}}: \underline{t} \in \mathbb I\}$ by
\begin{equation}
   {V}^{\rho}(F_{\underline{t}}(x):\underline{t} \in \mathbb I)\ :=\ \sup_{\underline{t}_1 \preceq \underline{t}_2 \preceq \ldots \preceq \underline{t}_L}\left(\sum_{l=1}^{L-1}|F_{\underline{t}_{l+1}}(x)-F_{\underline{t}_{l}}(x)|^{\rho}\right)^{1/{\rho}},  
\end{equation}
where the supremum is taken over all finite length increasing sequences $\underline{t}_1 \preceq \underline{t}_2 \preceq \ldots \preceq \underline{t}_L$.
 It is straight-forward to see that the finiteness of ${V}^{\rho}(F_{\underline{t}}(x):\underline{t} \in \mathbb I)$ implies that the limit $\lim_{\underline{t} \to \infty}F_{\underline{t}}(x)=\lim_{\min\{t_1,\ldots ,t_d\} \to \infty}F_{\underline{t}}(x)$ exists if $\mathbb I \cap [n,\infty)^d \neq \emptyset$ for all $n \in \mathbb N.$

\subsection*{Acknowledgements} I would like to thank my supervisor, Prof. Jim Wright, for suggesting the problem, guiding me throughout this project and providing meticulous feedback on earlier versions of the manuscript. I thank the referee for a careful reading of the manuscript and for valuable comments--in particular, for pointing out that our main result does not require the condition $s \leq q$ (see Remark \ref{ext}), which enabled us to extend the result to a broader range.

\section{Statement of the theorems}

 Let $T: L^{p,r}(Y) \to L^{q,s}(X)$ be a bounded linear operator with the operator norm $\|T\|_{L^{p,r} \to L^{q,s}}$. Assume that $Y = Y_1 \times \cdots \times Y_d$ is a product space of measure spaces $(Y_j, \nu_j)$ and $\{{\mathcal Y}^j_N\}_{N\in {\mathbb Z}}$ is a filtration of $Y_j$, $j=1,\ldots,d$. Consider the multiparameter truncations
$$
T_{{\underline{N}}}(f)(x) \ := \
T(f \mathbbm{1}_{{\mathcal Y}^1_{N^1} \times \cdots \times {\mathcal Y}^d_{N^d}})(x),
$$
where ${\underline{N}} = (N^1, \ldots, N^d) \in {\mathbb Z}^d$. It is convenient to write $\mathcal{Y}^1_{N^1} \times \cdots \times \mathcal{Y}^d_{N^d}$ as $\mathcal{Y}_{\underline{N}}$. The Christ--Kiselev argument extends to establish the following
multiparameter variational bound.

\begin{theorem}\label{multi}
    If $1 \leq p \leq r < \min\{q,s\} \leq \max\{q,s\} < \infty$ and $\rho>r$, then there exists a constant $A_{\rho, r}$ such that
    $$\bigl\|{V}^{\rho}(T_{\underline{N}}f:\underline{N} \in \mathbb Z^d)\bigr\|_{L^{q,s}(X)} \leq C_{\rho,p,q,r,s}\|f\|_{L^{p,r}(Y)},$$
    for all $f \in L^{p,r}(Y)$, where $C_{\rho,p,q,r,s}=A_{\rho,r}\|T\|_{L^{p,r}\to L^{q,s}}$.
\end{theorem}
\begin{remark}\label{ext}
While \cite{CKL} requires the condition $s \leq q$ in their proof, our method does \textit{not} require this condition. Consequently, we are able to extend Theorem \ref{CKL} to a broader range.
\end{remark}
Let us consider an application of Theorem \ref{multi}. Recall that C. Fefferman \cite{F} constructed a continuous function on the 2-torus ${\mathbb T}^2$ with Fourier coefficients ${\mathbf{a}} = \{a_{m,n}\}_{(m,n) \in \mathbb Z^2}$ such that the multiparameter partial sums in (\ref{mps}),
$$
S_{M,N}({\mathbf{a}})(\theta,\omega) \ = \ \sum_{|m|\le M} \sum_{|n|\le N} a_{m,n} {\rm e}^{-2\pi i (m\theta + n\omega)} \ \ {\rm diverge \ for \ every} \  (\theta, \omega) \in {\mathbb T}^2.
$$
A simple application of our main theorem gives the bound
\begin{equation}\label{V-SMN}
\|{V}^{\rho} (S_{M,N}({\mathbf{a}}): (M,N) \in {\mathbb N}^2 )\|_{L^{p'}({\mathbb T}^2)} \ \le \ C_{\rho, p} \, \|{\mathbf{a}}\|_{\ell^p({\mathbb Z}^2)}
\end{equation}
whenever $1\le p < 2$ and $p< \rho$.
As a consequence, we see that if ${\mathbf{a}} \in \ell^p({\mathbb Z}^2), 1\le p < 2$, then
$$
{V}^{\rho}(S_{M,N}(\mathbf{a})(\theta,\omega):(M,N) \in \mathbb N^2) < \infty \,\,\,\, {\rm a.e. } \,\,(\theta,\omega) \in \mathbb T^2.
$$
For these $(\theta, \omega) \in \mathbb T^2$,
$$
\lim_{\min\{M,N\} \to \infty} S_{M,N}({\mathbf{a}})(\theta,\omega) \ \ {\rm exists \ for \ almost \ every } \ (\theta,\omega)  \in {\mathbb T}^2.
$$
The above computation confirms the following well-known fact. Although the sequence of the Fourier coefficients of a continuous function lies in $\ell^2$, there are continuous functions whose Fourier coefficients do {\it not} lie in any $\ell^p, p<2$; see for example \cite{Z}.

We also observe that the Christ--Kiselev argument extends from linear to multilinear operators. Instead of starting with a bounded linear operator $T:L^{p,r}(Y) \to L^{q,s}(X)$, we begin with a bounded multilinear operator $T:L^{p_1,r_1}(Y^1) \times L^{p_2,r_2}(Y^2) \times \ldots \times L^{p_k,r_k}(Y^k) \to L^{q,s}(X)$ between measure spaces $(Y^1,\nu_1), \ldots, (Y^k,\nu_k)$  and $(X, \mu)$ with the operator norm $\|T\|_{\prod_{i=1}^{k}L^{p_i,r_i} \to L^{q,s}}$. Furthermore, suppose that for each $i=1,\ldots,k$, $Y^i=Y^i_1 \times \ldots \times Y^i_d$ is a product space of measure spaces $(Y^i_j,\nu_{i,j})$ and $\nu_i=\nu_{i,1} \otimes \ldots \otimes \nu_{i,d}$ is the product measure. We assume that each measure space $Y^i_j$ has a filtration $\{\mathcal{Y}^{i,j}_{N}\}_{N \in \mathbb Z}$. If $\underline{N}^i=(N^{i,1}, \ldots, N^{i,d}) \in \mathbb Z^d$ and $\mathcal{Y}^{i}_{\underline{N}^i}=\mathcal{Y}^{i,1}_{N^{i,1}} \times \ldots \times \mathcal{Y}^{i,d}_{N^{i,d}}$, we define the multilinear-multiparameter truncations by 
\begin{align*}
    T_{\underline{N}^1, \ldots, \underline{N}^k}(f_1,\ldots,f_k) &:= T(f_1\mathbbm{1}_{\mathcal{Y}^1_{\underline{N}^1}}, \ldots, f_k\mathbbm{1}_{\mathcal{Y}^k_{\underline{N}^k}}) \\
    &= T(f_1\mathbbm{1}_{\mathcal{Y}^{1,1}_{N^{1,1}} \times \ldots \times \mathcal{Y}^{1,d}_{N^{1,d}}}, \ldots , f_k\mathbbm{1}_{\mathcal{Y}^{k,1}_{N^{k,1}} \times \ldots \times \mathcal{Y}^{k,d}_{N^{k,d}}}).
\end{align*}
We have the following extension of Theorem \ref{multi}.

\begin{theorem}\label{multilinpar}
    If $1 \leq p_i \leq r_i < \min\{q,s\} \leq \max\{q,s\} < \infty$ and $\rho>r_i$ for all $i=1,\ldots,k$, then there is a constant $A_{\rho,r}$ such that
   \begin{equation}\label{mmrvar}
   \|{V}^{\rho}(T_{\underline{N}^1, \ldots, \underline{N}^k}(f_1,\ldots,f_k):\underline{N}^1, \ldots, \underline{N}^k \in \mathbb Z^d)\|_{L^{q,s}(X)}  \leq C_{\rho,p,q,r,s} \prod_{i=1}^{k} \|f_i\|_{L^{p_i,r_i}(Y^i)},
  \end{equation}
    where $C_{\rho,p,q,r,s}=A_{\rho,r}\|T\|_{\prod_{i=1}^{k}L^{p_i,r_i} \to L^{q,s}}$, $p=(p_1,p_2,\ldots,p_k)$ and $r=(r_1,r_2,\ldots,r_k)$.
\end{theorem}

The $\rho$-variational operator in (\ref{mmrvar}) is 
\begin{equation}
    \begin{split}
   {V}^{\rho}(T&_{\underline{N}^1, \ldots, \underline{N}^k}(f_1,\ldots,f_k)(x):\underline{N}^1, \ldots, \underline{N}^k \in \mathbb Z^d)  \\   
   &:=\sup_{\mathcal{S}}\Bigl(\sum_{l=1}^{L-1}|T(f_1\mathbbm{1}_{\mathcal{Y}^1_{\underline{N}^1_{l+1}}},\ldots,f_k\mathbbm{1}_{\mathcal{Y}^k_{\underline{N}^k_{l+1}}})(x)-T(f_1\mathbbm{1}_{\mathcal{Y}^1_{\underline{N}^1_{l}}},\ldots,f_k\mathbbm{1}_{\mathcal{Y}^k_{\underline{N}^k_{l}}})(x)|^{\rho}\Bigr)^{1/{\rho}},
    \end{split}
\end{equation}
where the supremum is taken over the set $\mathcal{S}$ consisting of all sequences $(\underline{N}^1_1,\ldots, \underline{N}^k_1) \preceq \ldots \preceq (\underline{N}^1_L,\ldots, \underline{N}^k_L)$. For each $l=1,\ldots,L$ and $i=1, \ldots, k$, $\underline{N}^i_l=(N^{i,1}_l,\ldots, N^{i,d}_l)$ is a $d$-tuple of integers, so any increasing sequence in $\mathcal{S}$ is unpacked to $kd$ increasing sequences of integers. Thus, the operator in (\ref{mmrvar}) can be viewed as the $\rho$-variational operator for a $kd$-paramater family of operators.
\begin{remark}
   While Theorem \ref{multi} is a special case of Theorem \ref{multilinpar}, a direct proof for the multilinear operator introduces heavy notation. To avoid cumbersome notation, we first establish Theorem \ref{multi} and subsequently use it to deduce Theorem \ref{multilinpar}.
\end{remark}

\begin{remark}
 It is important to remark that Christ and Kiselev in \cite{CK} also considered multilinear extensions of their maximal theorem but the multilinear operators they were interested in are of a particular kind arising from their main application to the growth of generalised eigenfunctions of the 1-dimensional Schr\"odinger operator with potential. Their argument establishing Theorem \ref{CK-thm} easily extends to bound these particular mulitilinear operators {\it but} for their application, Christ and Kiselev needed substantial refinements on how the multilinear operator norm depends on the degree of multilinearity. This requires a delicate analysis that one does not need for the proof of Theorem \ref{multilinpar} above.
\end{remark}

\section{Preliminaries}
We begin by noting the following lemma for Lorentz spaces that is proved in \cite{CKL}. 
\begin{lemma}\label{Lorlemma}(\cite{CKL})
    Let $F,G$ be disjoint measurable sets in $Y$ and let $f,g$ be measurable functions on $Y$. If $p \leq r < \infty$, then
    \begin{equation}\label{revminklor}
        \|f \mathbbm{1}_{F}+ g \mathbbm{1}_{G}\|^r_{L^{p,r}(Y)} \geq \|f \mathbbm{1}_{F}\|^r_{L^{p,r}(Y)} + \|g \mathbbm{1}_{G}\|^r_{L^{p,r}(Y)}.
    \end{equation}
    \end{lemma}

We say that a measure space $(Y, \nu)$ is divisible if for every measurable set $S \subseteq Y$ and $t \in [0,\nu(S)]$ there is a measurable set $S' \subseteq S$ such that $\nu(S')=t$. Let $f \in L^{p,r}(Y)$ such that $\|f\|^r_{L^{p,r}(Y)}=1$. Define the map
\begin{equation}\label{defnlambda}
    \lambda (S)=\|f \mathbbm{1}_{S}\|^r_{L^{p,r}(Y)}
\end{equation}
 on measurable subsets $S \subseteq Y$. Although $\lambda$ is not a measure on $Y$, it still satisfies the set theoretical concept of divisibility if $(Y, \nu)$ is divisible; that is, for every $t \in [0, \lambda(S)]$ there is a $\nu$-measurable subset $S' \subseteq S$ such that $\lambda(S')=t$. See Remark \ref{divisiblity}. We now turn to a modification of the Christ--Kiselev decomposition lemma, which plays an important role in our main results.
    \begin{lemma}\label{decomp}
    Let $(Y, \nu)$ be a divisible measure space and $\{\mathcal{Y}_N: N \in \mathbb Z\}$ be a filtration of Y. For $f \in L^{p,r}(Y)$ with $\|f\|^r_{L^{p,r}(Y)}=1$ we define $\lambda$ as in (\ref{defnlambda}). If $1 \leq p \leq r < \infty$, then there is a collection $\{B^m_j:m \in \mathbb N, 1 \leq j \leq 2^m\}$ satisfying the following conditions:
    \begin{enumerate}
        \item For each $m$, $\{B^{m}_{j}:1 \leq j \leq 2^m\}$ partitions Y into disjoint measurable subsets.
        \item Each $B^{m}_{j}$ is a union of precisely two sets $B^{m+1}_{j_1}$ and $B^{m+1}_{j_2}$.
        \item $\lambda(B^{m}_{j}) \leq 2^{-m}$ for all $m$, $j$.
        \item If $\mathcal{Y}_{M}, \mathcal{Y}_{N}$ are filtration elements, then there exist subsets $\mathbb N^1$ and $\mathbb N^2$ of $\mathbb N$ such that $\mathcal{Y}_{N}\setminus \mathcal{Y}_{M}$ may be decomposed, modulo $\lambda$-null sets, as an empty, finite, or countably infinite union 
        \begin{equation}\label{maindec}
        \mathcal{Y}_{N}\setminus \mathcal{Y}_{M}=\big(\bigcup_{m \in \mathbb N^1} B^{m}_{j(m)}\big) \cup \big(\bigcup_{m' \in \mathbb N^2} B^{m'}_{j'(m')}\big).
        \end{equation}
    \end{enumerate}
    \end{lemma}
A version of this result is proved in \cite{CKL}, but we give a different proof here which is inspired from \cite{CK}. We shall prove Lemma \ref{decomp} using the classical dyadic martingale structure on the unit interval.
\subsection{Classical dyadic intervals}
 Let $I_0=[0,1)$ and let $\mathcal{D}$ denote the collection of all dyadic intervals of $[0,1)$; that is, 
$$\mathcal{D}=\{I^{m}_{j}:=[(j-1)/2^m,j/2^m): m\geq 0, j=1,\ldots,2^m\}.$$
For a fixed $m \in \mathbb N$, we can write $[0,1)$ as a disjoint union $[0,1)=\cup_{j=1}^{2^m}I^m_j$. Clearly, $I^m_j=I^{m+1}_{2j-1} \cup I^{m+1}_{2j}$. We call $I^m_j$ the \textit{parent} of the \textit{left sibling} $I^{m+1}_{2j-1}$ and the \textit{right sibling} $I^{m+1}_{2j}$ and we write $I^{m}_{j}=P(I^{m+1}_{2j-1})=P(I^{m+1}_{2j})$. If $\mathcal{L}$ is the collection of all left siblings and $\mathcal{R}$ is the collection of all right siblings, then $\mathcal{L}=\{I^{m}_{j} \in \mathcal{D}: j {\rm \ is \ odd} \}$, $\mathcal{R}=\{I^{m}_{j} \in \mathcal{D}: j {\rm \  is \ even}\}$ and $\mathcal{D}$ is partitioned into $\mathcal{D}=\mathcal{L} \cup \mathcal{R}$. We aim to decompose any interval $[a,b)$ of $[0,1)$ into a union of the dyadic intervals. Let $\mathcal{D}_{a,b}$ be the collection of all dyadic intervals $I \subseteq [a,b)$ whose parent $P(I)$ leaks out of $[a,b)$; that is,
\begin{equation}
  \mathcal{D}_{a,b}:=\{I^{m}_{j} \in \mathcal{D}:I^{m}_{j} \subseteq [a,b), P(I^m_{j}) \not\subseteq [a,b)\}.  
\end{equation}
When $a = 0$ and $b = 1$, this is understood as $\mathcal{D}_{0,1} = \{[0, 1/2), [1/2, 1)\}$. This countable subcollection satisfies the following important properties.

\begin{lemma}\label{Dab}
    If $[a,b) \subseteq [0,1)$ and $\mathcal{D}_{a,b}$ is defined as above, then we have the following:
    \begin{enumerate}
        \item If $I$, $I' \in \mathcal{D}_{a,b}$ and $I \neq I'$, then $I \cap I' = \emptyset$;
        \item $(a,b) \subseteq \cup_{I \in \mathcal{D}_{a,b}}I \subseteq [a,b)$;
        \item For every $m \in \mathbb N$ there is at the most one $I^m_{j} \in \mathcal{L}$ such that $I^m_{j} \in \mathcal{D}_{a,b}$. The same conclusion hold for the right siblings.
    \end{enumerate}
\end{lemma}

\begin{proof}
    \begin{enumerate}\item Suppose that $I$, $I' \in \mathcal{D}_{a,b}$, $I \neq I'$ and $x \in I \cap I'$. Let $I=I^m_{j}$ and $I'=I^n_{l}.$ As the two sets have non-empty intersection, $m \neq n$. Let $m>n$. We notice that the $(m-n)^{\rm th}$ parent of $I^m_j$, that is, $P^{m-n}(I^m_j)$ is a dyadic interval at the level $n$ and $x \in P^{m-n}(I^m_j)$. We also know that $x \in I^n_l$; this is possible only when $P^{m-n}(I^m_j)=I^n_l$, therefore, the parent of $I^m_j$ is also in $\mathcal{D}_{a,b}$, a contradiction to $I^m_j \in \mathcal{D}_{a,b}$. This shows that two distinct elements of $\mathcal{D}_{a,b}$ has empty intersection.

        \item We can assume that $[a, b) \neq [0, 1)$, as the excluded case is trivial. Fix $x \in (a,b)$. Then there is $m \in \mathbb N$ such that $x-a>1/2^m$ and $b-x>1/2^m$ and a unique $j \in \{1,2,\ldots,2^m\}$ such that $x \in I^m_j$. We claim that $I^m_j \subseteq [a,b)$. If the left endpoint of $I^m_j$, $(j-1)/2^m<a$, then $(j-1)/2^m <a<x<j/2^m$, or $x-a \leq 1/2^m$, a contradiction. Similarly, we must have $j/2^m<b$; this shows that $I^m_{j}$ lies entirely in $[a,b).$
 Finally, take $I=P^{k}(I^m_j)$ so that $P^{k+1}(I^m_j) \not\subseteq [a,b)$. As $P^m(I^m_{j})=[0,1)$, $k$ is well defined and finite. This proves the first inclusion. The other inclusion holds trivially because of the definition of $\mathcal{D}_{a,b}$.

        \item Suppose that $I^m_j, I^m_{j'} \in \mathcal{L} \cap \mathcal{D}_{a,b}$ and $j \neq j'$. As $I^m_j, I^m_{j'} \in \mathcal{L}$, $j$ and $j'$ are odd, $|j-j'| \geq 2$. Suppose that $j'>j$; in other words, $j' \geq 2+j$ or $j'-1 \geq j+1$. Therefore, 
        $$a \leq (j-1)/2^m <j/2^m<(j+1)/2^m \leq (j'-1)/2^m <j'/2^m<b.$$ 
        This shows that the right sibling $I^m_{j+1}=[j/2^m,(j+1)/2^m) \subseteq [a,b)$ and hence $I^m_j \cup I^m_{j+1}=P(I^m_j) \subseteq [a,b)$, a contradiction. Thus, at each level $m$, there is at the most one left sibling in $\mathcal{D}_{a,b}$. We can apply the same argument for right siblings.
    \end{enumerate}
\end{proof}

\begin{remark}\label{phianddecomp}
    Lemma \ref{Dab} allows us to decompose $[a,b)$ as a union of left and right siblings  
    $$[a,b)=\left(\bigcup_{s=1}^{\infty}I^{m_s}_{j(m_s)}\right) \cup \left(\bigcup_{t=1}^{\infty}I^{m'_t}_{j'(m'_t)}\right) \cup \{a\},$$
    where $m_1<m_2<\ldots$ and $I^{m_s}_{j(m_s)}$ are elements of $\mathcal{L}$, $m'_1<m'_2<\ldots$ and $I^{m'_t}_{j'(m'_t)}$ are right siblings.
    If we denote the set $\{m_1,m_2,\ldots\}$ by $\mathbb N^{1}$ and $\{m'_1,m'_2,\ldots\}$ by $\mathbb N^2$, then we can write the above decomposition as
    \begin{equation}\label{[a,b)}
      [a,b)=\big( \cup_{m \in \mathbb N^1}I^m_{j(m)}\big) \cup \big(\cup_{m' \in \mathbb N^2}I^{m'}_{j'(m')}\big) \cup \{a\},  
    \end{equation}
    where $I^{m}_{j(m)}, m \in \mathbb N^1$ and $I^{m'}_{m'(j')}, m' \in \mathbb N^2$ are left and right siblings respectively. We also note that if $a=0$, then 
    \begin{equation}\label{[0,b)}
     [0,b)=\cup_{m \in \mathbb N^1}I^m_{j(m)},
    \end{equation}
 where the decomposition has only left siblings.
 \end{remark}

\begin{proof}[Proof of Lemma \ref{decomp}]\label{lemproof}
 In \cite{CK}, a mapping $\varphi:Y \to [0,1]$ was constructed that satisfies $\varphi^{-1}([0,\lambda(Y_N)))=Y_N$, modulo $\lambda$-null sets, $N \in \mathbb Z$ and $\lambda(\varphi^{-1}([0,t)))=t$ for all $t \in [0,1]$. For the convenience of the reader we give this construction here. 
 Since the divisibility of $(Y, \nu)$ implies the set theoretical divisibility of $(Y, \lambda)$, we can extend the filtration $\{\mathcal{Y}_N\}_{N \in \mathbb Z}$ to $\{Z_t\}_{t \in \mathcal{A}}$ with respect to $\lambda$, where $\mathcal{A}$ is a dense subset of $\mathbb R$ containing $\mathbb Z$ such that $Z_{t}=\mathcal{Y}_t$ for all $t \in \mathbb Z$. Indeed, if we fix an annulus $\mathcal{Y}_N\setminus \mathcal{Y}_{N-1}$, we can find a $\nu$-measurable set $Z$ such that $\mathcal{Y}_{N-1} \subseteq Z \subseteq \mathcal{Y}_{N}$ and $\lambda(Z)=(\lambda(\mathcal{Y}_{N-1})+\lambda(\mathcal{Y}_{N}))/2$. We repeat this procedure to find $\nu$-measurable sets $Z_1$ and $Z_2$ such that $\mathcal{Y}_{N-1} \subseteq Z_1 \subseteq Z$, $Z \subseteq Z_2 \subseteq \mathcal{Y}_{N}$ and $\lambda(Z_1)=(\lambda(\mathcal{Y}_{N-1})+\lambda(Z))/2$ and $\lambda(Z_2)=(\lambda(Z)+\lambda(\mathcal{Y}_{N}))/2$. This gives us the required extension.
 We define $\varphi$ as 
$$\varphi(y)=\sup_{y \notin Z_s}\lambda(Z_s)=\inf_{y \in Z_s}\lambda(Z_s).$$
We first prove that $\lambda(\varphi^{-1}([0,t)))=t$ when $t=\lambda(Z_s)$ for some $s \in \mathcal{A}$. Note that
$$\varphi^{-1}([0,t))=\varphi^{-1}([0,\lambda(Z_s)))=\{y \in Y: 0 \leq \sup_{y \notin Z_w} \lambda(Z_w) < \lambda(Z_s)\}.$$
Suppose that $y \in \varphi^{-1}([0,\lambda(Z_s)))$. If $y \notin Z_s$, then $\lambda (Z_s) \leq \sup_{y \notin Z_w} \lambda(Z_w) < \lambda(Z_s),$ a contradiction. Therefore, $\varphi^{-1}([0,\lambda(Z_s))) \subseteq Z_s.$ Suppose that $y \in Z_s$. Then for any $Z_w$ with $y \notin Z_w$, we must have $Z_w \subseteq Z_s$, because $\{Z_t:t \in \mathcal{A}\}$ is a filtration. Therefore, $\sup_{y \notin Z_w} \lambda(Z_w) \leq \lambda(Z_s)$ and hence $Z_s \subseteq \varphi^{-1}([0,\lambda(Z_s)]).$ Therefore, if $t=\lambda(Z_s)$ for some $s \in \mathcal{A}$, then $$\lambda(\varphi^{-1}([0,t))) \le \lambda (Z_s) \le \lambda(\varphi^{-1}([0,t])).$$ 
Using the properties of Lorentz norm, we have $\lambda(\varphi^{-1}(\{t\}))=0$ for all $t \in [0,1]$. This ensures that $\lambda(\varphi^{-1}([0,t)))= \lambda(\varphi^{-1}([0,t]))=\lambda(Z_s)=t,$ when $t=\lambda(Z_s)$. Note that this argument also proves our first claim that $\varphi^{-1}([0,\lambda(Y_N)))=Y_N$, modulo $\lambda$-null sets.

Finally, for any $t \in [0,1]$, there is a decreasing sequence $\{t_n\}$ such that $t_n \to t$ and $t_n=\lambda(Z_{s_n})$. Therefore, 
$$\lambda(\varphi^{-1}([0,t)))=\lim_{n \to \infty}\lambda(\varphi^{-1}([0,t_{n})))= \lim_{n \to \infty} \lambda(Z_{s_n})=\lim_{n \to \infty} t_n=t.$$
 We define 
$$B^m_{j}=\varphi^{-1}(I^{m}_{j})=\varphi^{-1}([(j-1)/2^m,j/2^m)).$$ 
The properties of dyadic intervals trivially implies the first two statements of the lemma. As $p \leq r$, by (\ref{revminklor}) of Lemma \ref{Lorlemma} we have
$$\lambda(B^{m}_{j})=\lambda\bigl(\varphi^{-1}\bigl(\bigl[\frac{j-1}{2^m},\frac{j}{2^m}\bigr)\bigr)\bigr) \le \lambda\bigl(\varphi^{-1}\bigl(\bigl[0,\frac{j}{2^m}\bigr)\bigr)\bigr)-\lambda\bigl(\varphi^{-1}\bigl(\bigl[0,\frac{j-1}{2^m}\bigr)\bigr)\bigr)=\frac{1}{2^m}.$$
Finally, modulo $\lambda$-null sets, we have
\begin{align*}
    \mathcal{Y}_{N} \setminus \mathcal{Y}_{M} &=  \varphi^{-1}([0, \lambda(\mathcal{Y}_{N}))) \setminus \varphi^{-1}([0, \lambda(\mathcal{Y}_{M}))) \\
  &= \varphi^{-1}([0, \lambda(\mathcal{Y}_{N}))\setminus [0, \lambda(\mathcal{Y}_{M}))) \\
  &= \varphi^{-1}([\lambda(\mathcal{Y}_{M}),\lambda(\mathcal{Y}_{N}))). 
\end{align*}
We use Remark \ref{phianddecomp} to decompose the interval $[\lambda(\mathcal{Y}_{M}),\lambda(\mathcal{Y}_{N}))$ as a union of left and right siblings so that
\begin{align*}
    \mathcal{Y}_{N} \setminus \mathcal{Y}_{M} &= \varphi^{-1}([\lambda(\mathcal{Y}_{M}),\lambda(\mathcal{Y}_{N}))) \\
    &= \varphi^{-1}\Big[ \big( \cup_{m \in \mathbb N^1}I^m_{j(m)}\big) \cup \big(\cup_{m' \in \mathbb N^2}I^{m'}_{j'(m')}\big) \cup \{\lambda(\mathcal{Y}_{M})\} \Big] \\
    &=  \big( \cup_{m \in \mathbb N^1}B^m_{j(m)}\big) \cup \big(\cup_{m' \in \mathbb N^2}B^{m'}_{j'(m')}\big) \cup \varphi^{-1}(\{\lambda(\mathcal{Y}_{M})\}),
\end{align*}
where $\mathbb N^1$ and $\mathbb N^2$ are subsets of the set of natural numbers. 
\end{proof}
\begin{remark}\label{divisiblity}
Note that we proved Lemma \ref{decomp} for divisible measure spaces. As observed in \cite{CK}, we can, without loss of generality, assume that for every $j=1,\ldots,d$, $(Y_j, \nu_j)$ is divisible and prove Theorem \ref{multi}. This is because divisibility can be achieved, for $j=1,\ldots,d$, by replacing $Y_j$ by $Y_j \times [0,1]$, $\nu_j$ by the product of $\nu_j$ with the Lebesgue measure on $[0,1]$, $T$ by $T \circ \pi$ where $\pi f(y_1,\ldots,y_d)=\int_{[0,1]^d}f((y_1,s_1)\ldots (y_d,s_d))ds_1 \ldots ds_d$ and boundedness of the variational operator ${V}^{\rho}(T_{\underline{N}}f:\underline{N} \in \mathbb Z^d)$ is then implied by the boundedness of the operator ${V}^{\rho}((T_{\underline{N}} \circ \pi)f: \underline{N} \in \mathbb Z^d)$.
\end{remark}

\section{Proof of the main theorems}
 We first simplify the variational operator ${V}^{\rho}(T_{\underline{N}}f:\underline{N} \in \mathbb Z^d)$ given in the statement of Theorem \ref{multi}. Note that if $\underline{M}, \underline{N} \in \mathbb R^d$ such that $\underline{M} \preceq \underline{N}$, then $\mathcal{Y}_{\underline{M}} \subseteq \mathcal{Y}_{\underline{N}}$. The linearity of $T$ allows us to write the difference as
 $$T(f\mathbbm{1}_{\mathcal{Y}_{\underline{N}}})-T(f\mathbbm{1}_{\mathcal{Y}_{\underline{M}}})=T(f(\mathbbm{1}_{\mathcal{Y}_{\underline{N}}}-\mathbbm{1}_{\mathcal{Y}_{\underline{M}}}))=T(f\mathbbm{1}_{\mathcal{Y}_{\underline{N}}\setminus \mathcal{Y}_{\underline{M}}}).$$   Therefore the $\rho$-variational operator is reduced to
$$ {V}^{\rho}(T_{\underline{N}}f(x):\underline{N} \in \mathbb Z^d)=\sup_{\underline{N}_1 \preceq \underline{N}_2 \preceq \ldots 
 \preceq \underline{N}_L}\left(\sum_{l=1}^{L-1}|T(f\mathbbm{1}_{\mathcal{Y}_{\underline{N}_{l+1}}\setminus \mathcal{Y}_{\underline{N}_{l}}})(x)|^{\rho}\right)^{1/{\rho}}.$$
To avoid cumbersome notation, we prove Theorem \ref{multi} in the two-parameter case. The extension to $d$ parameters is straight-forward. In the two-parameter case, that is, when $d=2$, the difference set $\mathcal{Y}_{\underline{N}_{l+1}}\setminus \mathcal{Y}_{\underline{N}_{l}}$ can be expressed as a disjoint union of three sets $\mathcal{Y}^{1}_{N^{1}_{l}} \times (\mathcal{Y}^{2}_{N^{2}_{l+1}} \setminus \mathcal{Y}^{2}_{N^{2}_{l}})$, $(\mathcal{Y}^{1}_{N^{1}_{l+1}} \setminus \mathcal{Y}^{1}_{N^{1}_{l}}) \times (\mathcal{Y}^{2}_{N^{2}_{l+1}} \setminus \mathcal{Y}^{2}_{N^{2}_{l}})$ and $(\mathcal{Y}^{1}_{N^{1}_{l+1}} \setminus \mathcal{Y}^{1}_{N^{1}_{l}}) \times \mathcal{Y}^{2}_{N^{2}_{l}}$.
%We observe an important fact that in each of the sets there is atleast one set which is a difference of originally given filtration. 
We treat each region separately and the treatment is the same for all regions, so it is sufficient to prove the theorem for one of the sets and the corresponding $\rho$-variation operator, say ${\widetilde{V}^{\rho}}$, where
\begin{equation}\label{final}
{\widetilde{V}^{\rho}}(T_{\underline{N}}f(x):\underline{N} \in \mathbb Z^2) \ : = \ \sup_{(N^{1}_{1},N^{2}_{1}) \preceq \dots \preceq (N^{1}_{L},N^{2}_{L})}\left(\sum_{l=1}^{L-1}|T(f\mathbbm{1}_{\mathcal{Y}^{1}_{N^{1}_{l}} \times (\mathcal{Y}^{2}_{N^{2}_{l+1}} \setminus \mathcal{Y}^{2}_{N^{2}_{l}})})(x)|^{\rho}\right)^{1/{\rho}}.
\end{equation}

Let $f \in L^{p,r}(Y)$ be such that $\|f\|^r_{L^{p,r}(Y)}=1$; here $Y=Y_1 \times Y_2$. We define the mappings $\lambda_1$ and $\lambda_2$ on measurable subsets of $Y_1$ and $Y_2$ respectively by 
\begin{equation}\label{2meas}
    \lambda_1(S):=\|f \mathbbm{1}_{S \times Y_2}\|^r_{L^{p,r}(Y)} ; \,\,\,\,\, \lambda_2(U):=\|f \mathbbm{1}_{Y_1 \times U}\|^r_{L^{p,r}(Y)}.
\end{equation}
    We construct functions $\varphi_1:Y_1 \to [0,1]$ and $\varphi_2:Y_2 \to [0,1]$ as in Lemma \ref{decomp} to obtain collections of subsets $\{B^m_j:m \in \mathbb N,1 \leq j \leq 2^m\}$ in $Y_1$ and $\{C^n_i:n \in \mathbb N,1 \leq i \leq 2^n\}$ in $Y_2$ satisfying the conditions of Lemma \ref{decomp}. For fixed $m,n \in \mathbb N$, the collection $\{B^m_j \times C^n_i:1\leq j \leq 2^m,1 \leq i \leq 2^n\}$ forms a grid in $Y_1 \times Y_2$. Because $p \leq r$, using (\ref{revminklor}) of Lemma \ref{Lorlemma}, we have
    $$1 =\|f\|^r_{L^{p,r}(Y)} = \|\sum_{j=1}^{2^m} \sum_{i=1}^{2^n} f \mathbbm{1}_{B^m_j \times C^n_i}\|^r_{L^{p,r}(Y)} \geq \sum_{j=1}^{2^m} \sum_{i=1}^{2^n} \|f \mathbbm{1}_{B^m_j \times C^n_i}\|^r_{L^{p,r}(Y)},$$ that is,
 \begin{equation}\label{equality1}
      \sum_{j=1}^{2^m} \sum_{i=1}^{2^n} \|f \mathbbm{1}_{B^m_j \times C^n_i}\|^r_{L^{p,r}(Y)} \leq 1.
    \end{equation}
 We now give the proof of the main theorem.
\begin{proof}[Proof of Theorem \ref{multi}]
Let $f \in L^{p,r}(Y)$ be such that $\|f\|^{r}_{L^{p,r}(Y)}=1$. Fix a finite increasing sequence $(N^1_1,N^2_1) \preceq \cdots \preceq (N^1_L,N^2_L)$ in $\mathbb Z^2$. Using Lemma \ref{decomp} for $\lambda_1$ and $\lambda_2$ defined in (\ref{2meas}), we decompose the sets in (\ref{final}) as $$\mathcal{Y}^1_{N^1_l}=\varphi_1^{-1}([0,\lambda_1(\mathcal{Y}^1_{N^1_l}))) ; \,\, \mathcal{Y}^2_{N^2_{l+1}} \setminus \mathcal{Y}^2_{N^2_{l}}=\varphi_2^{-1}([\lambda_2(\mathcal{Y}^2_{N^2_{l+1}}), \lambda_2(\mathcal{Y}^2_{N^2_{l}}))), \,\,\,l=1,\ldots,L.$$ Now, the dyadic decomposition (\ref{[a,b)}) and (\ref{[0,b)}) in Remark \ref{phianddecomp} allows us to write
$$\mathcal{Y}^1_{N^1_{l}}=\bigcup_{m \in \mathbb N^1_{l}}B^m_{j(m,l)}; \ \ \ \ \ \ \  \mathcal{Y}^2_{N^2_{l+1}} \setminus \mathcal{Y}^2_{N^2_{l}}=\bigcup_{n \in \mathbb N^2_{l}}C^{n}_{i(n,l)} \cup \bigcup_{n' \in \mathbb N'^2_{l}}C^{n'}_{i'(n',l)}.$$
The notation $\mathbb N^1_{l}$, $\mathbb N^2_{l}$, $j(m,l)$ and $i(n,l)$ are used to show that these sets depend on $l$ as $l$ varies over $1,\ldots,L$. Therefore, 
$$\mathcal{Y}^1_{N^1_{l}} \times (\mathcal{Y}^2_{N^2_{l+1}} \setminus \mathcal{Y}^2_{N^2_{l}})=\bigcup_{m \in \mathbb N^1_{l}}B^m_{j(m,l)} \times \Bigl( \bigcup_{n \in \mathbb N^2_{l}}C^{n}_{i(n,l)} \cup \bigcup_{n' \in \mathbb N'^2_{l}}C^{n'}_{i'(n',l)}\Bigr).$$
    The right hand side splits into two identical parts which can be dealt with separately. So, without loss of generality, we assume that 
    \begin{equation}\label{decomposition}
     \mathcal{Y}^1_{N^1_{l}} \times (\mathcal{Y}^2_{N^2_{l+1}} \setminus \mathcal{Y}^2_{N^2_{l}})=\bigcup_{m \in \mathbb N^1_{l}}B^m_{j(m,l)} \times \bigcup_{n \in \mathbb N^2_{l}}C^{n}_{i(n,l)}.   
    \end{equation}
 By the linearity and boundedness of $T$ we have
$$T(f\mathbbm{1}_{\mathcal{Y}^1_{N^1_{l}} \times (\mathcal{Y}^2_{N^2_{l+1}} \setminus \mathcal{Y}^2_{N^2_{l}})})= \sum_{m \in \mathbb N^1_l} \sum_{n \in \mathbb N^2_l}T(f \mathbbm{1}_{B^m_{j(m,l)} \times C^n_{i(n,l)}}),$$
where the series converges in $L^{q,s}(X)$ norm.
Therefore, for almost all $x \in X$,
\begin{align}
   & \Big(\sum_{l=1}^{L-1}|T(f\mathbbm{1}_{\mathcal{Y}^1_{N^1_{l}} \times (\mathcal{Y}^2_{N^2_{l+1}} \setminus \mathcal{Y}^2_{N^2_{l}})})(x)|^{\rho}\Big)^{1/{\rho}}  \nonumber \\
    &\leq \Big(\sum_{l=1}^{L-1}\Big(\sum_{m \in \mathbb {N}^1_{l}} \sum_{n \in \mathbb {N}^2_{l}} |T(f\mathbbm{1}_{B^m_{j(m,l)} \times C^{n}_{i(n,l)}})(x)|\Big)^{\rho}\Big)^{1/{\rho}} \nonumber\\ 
     &\leq \sum_{m=1}^{\infty}\sum_{n=1}^{\infty} \Big(\sum_{\substack{l:m \in \mathbb{N}^1_{l} \\ n \in \mathbb {N}^2_{l}}}|T(f\mathbbm{1}_{B^m_{j(m,l)} \times C^{n}_{i(n,l)}})(x)|^{\rho}\Big)^{1/{\rho}}.\label{finalest}
\end{align}
For fixed $(m,n) \in \mathbb N^2$, consider the set $A=\{l:m \in \mathbb{N}^1_{l}, n \in \mathbb{N}^2_{l}\}$. We claim that for every $l, l' \in A$ such that $l \neq l'$, we have $i(n,l) \neq i(n,l')$. Suppose that $l <l'$. Then the inclusions $Y^2_{N^2_l} \subseteq Y^2_{N^2_{l+1}} \subseteq Y^2_{N^2_{l'}} \subseteq Y^2_{N^2_{l'+1}}$ imply that $(Y^2_{N^2_{l+1}} \setminus Y^2_{N^2_{l}}) \cap (Y^2_{N^2_{l'+1}} \setminus Y^2_{N^2_{l'}})=\emptyset$. Because $C^n_{i(n,l)}$ and $C^n_{i(n,l')}$ lie in the decomposition of $(Y^2_{N^2_{l+1}} \setminus Y^2_{N^2_{l}})$ and $(Y^2_{N^2_{l'+1}} \setminus Y^2_{N^2_{l'}})$ respectively, the intersection $C^n_{i(n,l)} \cap C^n_{i(n,l')}=\emptyset$. These sets are inverse images of the dyadic intervals $I^n_{i(n,l)}$ and $I^n_{i(n,l')}$ respectively under the map $\varphi_2$, and the disjointness of these sets implies that $i(n,l) \neq i(n,l')$, hence our claim.
This means that the set $\{(j(m,l),i(n,l)):l\in A\}$ consists of distinct terms. This allows us to dominate the sum in (\ref{finalest}) by 
$$\sum_{m=1}^{\infty}\sum_{n=1}^{\infty}\Bigg(\sum_{j=1}^{2^m}\sum_{i=1}^{2^n}|T(f\mathbbm{1}_{B^{m}_{j} \times C^{n}_{i}})(x)|^{\rho}\Bigg)^{1/{\rho}}.$$
Finally, as the above expression is independent of the sequence $(N^1_1,N^2_1) \preceq \cdots \preceq (N^1_L,N^2_L)$, we can take the supremum over all finite length increasing sequences to obtain the pointwise bound
\begin{equation}
    {\widetilde{V}^{\rho}}(T_{\underline{N}}f(x):\underline{N} \in \mathbb Z^2) \leq \sum_{m=1}^{\infty}\sum_{n=1}^{\infty}\Bigg(\sum_{j=1}^{2^m}\sum_{i=1}^{2^n}|T(f\mathbbm{1}_{B^{m}_{j} \times C^{n}_{i}})(x)|^{\rho}\Bigg)^{1/\rho}.
\end{equation}
Define
     $S_{m,n}f(x):=\Big(\sum_{j=1}^{2^m}\sum_{i=1}^{2^n}|T(f\mathbbm{1}_{B^{m}_{j} \times C^{n}_{i}})(x)|^{\rho}\Big)^{1/{\rho}}.$
     By the scaling property of the Lorentz quasi-norm, $\||g|^{1/{\rho}}\|_{L^{q,s}(X)} = \|g\|^{1/{\rho}}_{L^{q/{\rho},s/{\rho}}(X)}$, we have
$$\|S_{m,n}f\|_{L^{q,s}(X)} = \bigl\|\sum_{j=1}^{2^m} \sum_{i=1}^{2^n} |T(f\mathbbm{1}_{B_j^m \times C_i^n})|^{\rho}\bigr\|^{1/{\rho}}_{L^{q/{\rho},s/{\rho}}(X)}.$$
Let ${\rho} < \min\{q,s\}$. Then $\|\cdot\|_{L^{q/{\rho},s/{\rho}}(X)}$ is equivalent to a norm; that is, there is a norm $\|\cdot\|$ and constants $c=c(q,s,\rho)$ and $C=C(q,s,\rho)$ such that $c\|h\|_{L^{q/\rho,s/\rho}(X)}\le\|h\|\le C\|h\|_{L^{q/\rho,s/\rho}(X)}$.
Set $K=C/c$. We apply the triangle inequality for $\|\cdot\|$, return to
$\|\cdot\|_{L^{q/\rho,s/\rho}(X)}$, use the scaling property again on each term,
and finally use the boundedness of $T$. This gives
     \begin{align*}
         \|S_{m,n}f\|_{L^{q,s}(X)} 
        &\leq K^{1/{\rho}}\Bigl(\sum_{j=1}^{2^m} \sum_{i=1}^{2^n} \|T(f \mathbbm{1}_{B^m_j \times C^n_i})\|^{\rho}_{L^{q,s}(X)}\Bigr)^{1/{\rho}} \\
        &\leq  K^{1/{\rho}}\|T\|_{L^{p,r} \to L^{q,s}} \Big(\sum_{j=1}^{2^m}\sum_{i=1}^{2^n}\|f\mathbbm{1}_{B^{m}_{j} \times C^{n}_{i}}\|_{L^{p,r}(Y)}^{\rho}\Big)^{1/\rho}\\
        &\leq K^{1/{\rho}}\|T\|_{L^{p,r} \to L^{q,s}} \Big(\max_{j,i}\|f\mathbbm{1}_{B^{m}_{j} \times C^{n}_{i}}\|_{L^{p,r}(Y)}^{\rho-r}\Big)^{1/{\rho}}\Big(\sum_{j,i}\|f\mathbbm{1}_{B^{m}_{j} \times C^{n}_{i}}\|_{L^{p,r}(Y)}^{r}\Big)^{1/\rho} \\
        &\leq K^{1/{\rho}}\|T\|_{L^{p,r} \to L^{q,s}}\Big(\max_{j,i}\|f\mathbbm{1}_{B^{m}_{j} \times C^{n}_{i}}\|_{L^{p,r}(Y)}^{\rho-r}\Big)^{1/\rho};
     \end{align*}
     the last inequality is because $\sum_{j=1}^{2^m}\sum_{i=1}^{2^n}\|f\mathbbm{1}_{B^{m}_{j} \times C^{n}_{i}}\|_{L^{p,r}(Y)}^r \leq 1$ as observed in (\ref{equality1}). 
We note that
$$\left(\max_{j,i}\|f\mathbbm{1}_{B^{m}_{j} \times C^{n}_{i}}\|_{L^{p,r}(Y)}^{\rho-r}\right)^{1/\rho} \leq \bigl(\max_{j}\|f\mathbbm{1}_{B^m_j \times Y_2}\|^{\rho-r}_{L^{p,r}(Y)}\bigr)^{1/\rho} \leq 2^{-m(1/r-1/{\rho})}.$$
Similarly, we also have the bound 
$$\left(\max_{j,i}\|f\mathbbm{1}_{B^{m}_{j} \times C^{n}_{i}}\|_{L^{p,r}(Y)}^{\rho-r}\right)^{1/{\rho}} \leq 2^{-n(1/r-1/{\rho})}.$$
If $m \geq n$, then $2^{-m(1/r-1/{\rho})} \leq 2^{-n(1/r-1/{\rho})}$ and hence
$$\sum_{m \geq n}\left(\max_{j,i}\|f\mathbbm{1}_{B^{m}_{j} \times C^{n}_{i}}\|_{L^{p,r}(Y)}^{\rho-r}\right)^{1/\rho} \leq \sum_{n=1}^{\infty}\sum_{m=n}^{\infty}2^{-m(1/r-1/{\rho})} \leq A'_{\rho,r}.$$
If $n \geq m$, then $2^{-n(1/r-1/{\rho})} \leq 2^{-m(1/r-1/{\rho})}$ and hence
$$\sum_{n \geq m}\left(\max_{j,i}\|f\mathbbm{1}_{B^{m}_{j} \times C^{n}_{i}}\|_{L^{p,r}(Y)}^{\rho-r}\right)^{1/\rho} \leq \sum_{m=1}^{\infty}\sum_{n=m}^{\infty}2^{-n(1/r-1/{\rho})} \leq A'_{\rho,r}.$$
Combining the two inequalities above, we get
$$\|{V}^{\rho}(T_{N}f:N \in \mathbb Z)\|_{L^{q,s}(X)} \leq A_{\rho,r}\|T\|_{L^{p,r} \to L^{q,s}},$$ for every $f \in L^{p,r}(Y)$ with $\|f\|_{L^{p,r}(Y)}=1$ and $r< \rho < \min\{q,s\}$.  Finally, by the monotonicity of variation semi-norm, we have the theorem for all $r < \rho$.
\end{proof}
As we observed that the $\rho$-variational operator in Theorem \ref{multilinpar} is essentially a $kd$-parameter operator, the proof of the multilinear-multiparameter $\rho$-variation Christ--Kiselev theorem relies on the arguments of the multiparameter $\rho$-variational Christ--Kiselev theorem proved above. In the proof below, we apply Theorem \ref{multi} $k$ times simultaneously.
\begin{proof}[Proof of Theorem \ref{multilinpar}]
To simplify the notation, we prove the theorem for bilinear operators in the two-parameter case. The general case follows in the same manner. Without loss of generality, we consider $f_1 \in L^{p_1,r_1}(Y^1)$ and $f_2 \in L^{p_2,r_2}(Y^2)$ such that $\|f_1\|^{r_1}_{L^{p_1,r_1}(Y^1)} = \|f_2\|^{r_2}_{L^{p_2,r_2}(Y^2)}=1$. We recall that for $k=1,2$, $Y^k$ is the product space $Y^k=Y^k_1 \times Y^k_2$. Define the mappings $\lambda_{k,1}$ and $\lambda_{k,2}$ on measurable subsets of $Y^k_1$ and $Y^k_2$ respectively by 
$$\lambda_{k,1}(S):=\|f_k \mathbbm{1}_{S \times Y^k_2}\|^{r_k}_{L^{p_k,r_k}(Y^k)}; \,\,\,\,\, \lambda_{k,2}(U):=\|f_k \mathbbm{1}_{Y^k_1 \times U}\|^{r_k}_{L^{p_k,r_k}(Y^k)}.$$ Hence, we obtain collections $\{B^{m_k}_{j_k}:m_{k} \in \mathbb N,1 \leq j_{k} \leq 2^{m_k}\}$ in $Y^k_1$ and $\{C^{n_k}_{i_k}:n_{k} \in \mathbb N,1 \leq i_{k} \leq 2^{n_k}\}$ in $Y^k_2$ satisfying the conditions of Lemma \ref{decomp} and
\begin{equation}\label{ineqmultilin}
   \sum_{j_k=1}^{2^{m_k}} \sum_{i_{k}=1}^{2^{n_k}} \|f_k \mathbbm{1}_{B^{m_k}_{j_k} \times C^{n_k}_{i_k}}\|^{r_k}_{L^{p_k,r_k}(Y^k)} \leq 1. 
\end{equation}
We also recall that
\begin{equation*}
\begin{split}
{V}^{\rho}(T_{\underline{N}^1,\underline{N}^2}&(f_1,f_2)(x): \underline{N}^1, \underline{N}^2 \in \mathbb Z^d)  \\
&=\sup_{\mathcal{S}}\Big(\sum_{l=1}^{L-1}|T(f_1\mathbbm{1}_{\mathcal{Y}^1_{\underline{N}^1_{l+1}}},f_2\mathbbm{1}_{\mathcal{Y}^2_{\underline{N}^2_{l+1}}})(x)-T(f_1\mathbbm{1}_{\mathcal{Y}^1_{\underline{N}^1_{l}}},f_2\mathbbm{1}_{\mathcal{Y}^2_{\underline{N}^2_{l}}})(x)|^{\rho}\Big)^{1/{\rho}},
\end{split}
\end{equation*}
where the supremum is taken over the set $\mathcal{S}$ of all finite length increasing sequences $(\underline{N}^1_1,\underline{N}^2_1) \preceq \ldots \preceq (\underline{N}^1_L,\underline{N}^2_L)$.
Observe that for $k=1,2$,
\begin{equation*}
  \mathcal{Y}^k_{\underline{N}^k_{l+1}} = \mathcal{Y}^{k,1}_{{N}^{k,1}_{l+1}} \times \mathcal{Y}^{k,2}_{{N}^{k,2}_{l+1}} 
  = \Bigl(\mathcal{Y}^{k,1}_{{N}^{k,1}_{l}}\bigcup \, \Bigl(\mathcal{Y}^{k,1}_{{N}^{k,1}_{l+1}} \setminus \mathcal{Y}^{k,1}_{{N}^{k,1}_{l}}\Bigr)\Bigr) \times \Bigl(\mathcal{Y}^{k,2}_{{N}^{k,2}_{l}}\bigcup\,\Bigl(\mathcal{Y}^{k,2}_{{N}^{k,2}_{l+1}} \setminus \mathcal{Y}^{k,2}_{{N}^{k,2}_{l}}\Bigr)\Bigr).
\end{equation*}
Therefore, the bilinearity of $T$ breaks $T(f_1\mathbbm{1}_{\mathcal{Y}^1_{\underline{N}^1_{l+1}}},f_2\mathbbm{1}_{\mathcal{Y}^2_{\underline{N}^2_{l+1}}})$ into a sum of $16$ terms out of which one of the terms is $T(f_1\mathbbm{1}_{\mathcal{Y}^1_{\underline{N}^1_{l}}},f_2\mathbbm{1}_{\mathcal{Y}^2_{\underline{N}^2_{l}}})$, and hence the difference term above is a sum of $15$ (in general, $2^{kd}-1$) terms. Each of these terms can be treated separately so it is enough to treat ${\widetilde{V}^{\rho}}$, where 
 \begin{equation}
 \begin{split}
    {\widetilde{V}^{\rho}}(T_{\underline{N}^1,\underline{N}^2}&(f_1,f_2)(x): \underline{N}^1, \underline{N}^2 \in \mathbb Z^d)\\
    &= \sup_{\mathcal{S}}\Bigg(\sum_{l=1}^{L-1}|T(f_1\mathbbm{1}_{\mathcal{Y}^{1,1}_{N^{1,1}_l} \times (\mathcal{Y}^{1,2}_{N^{1,2}_{l+1}} \setminus \mathcal{Y}^{1,2}_{N^{1,2}_l})}, f_2\mathbbm{1}_{\mathcal{Y}^{2,1}_{N^{2,1}_l} \times (\mathcal{Y}^{2,2}_{N^{2,2}_{l+1}} \setminus \mathcal{Y}^{2,2}_{N^{2,2}_l})})(x)|^{\rho}\Bigg)^{1/{\rho}}.
    \end{split}
 \end{equation}
Indeed, we shall see that the other terms can be treated similarly, as it suffices to exploit cancellations corresponding to just one function, say $f_1$. For $k=1,2$, using the reduction leading to (\ref{decomposition}), we decompose these filtration elements and its differences as
\begin{equation}
  \mathcal{Y}^{k,1}_{N^{k,1}_{l}} \times (\mathcal{Y}^{k,2}_{N^{k,2}_{l+1}} \setminus \mathcal{Y}^{k,2}_{N^{k,2}_{l}})=\bigcup_{m_k \in \mathbb N^{k,1}_{l}}B^{m_k}_{j_k(m_k,l)} \times \bigcup_{n_k \in \mathbb N^{k,2}_{l}}C^{n_k}_{i_k(n_k,l)},  
\end{equation}
where $\mathbb N^{k,1}_{l}$ and $\mathbb N^{k,2}_l$ are subsets of the set of natural numbers and $\lambda_{k,1}(B^{m_k}_{j_k}) \leq 2^{-m_k}$ for all $j_k=1,\ldots,2^{m_k}$ and $\lambda_{k,2}(C^{n_k}_{i_k}) \leq 2^{-n_k}$ for all $i_k=1,\ldots,2^{n_k}$.
Observe that for $l=1, \ldots, L$, we have
    \begin{align*}
        &T(f_1\mathbbm{1}_{\mathcal{Y}^{1,1}_{N^{1,1}_{l}} \times (\mathcal{Y}^{1,2}_{N^{1,2}_{l+1}} \setminus \mathcal{Y}^{1,2}_{N^{1,2}_{l}})},f_2\mathbbm{1}_{\mathcal{Y}^{2,1}_{N^{2,1}_{l}} \times (\mathcal{Y}^{2,2}_{N^{2,2}_{l+1}} \setminus \mathcal{Y}^{2,2}_{N^{2,2}_{l}})}) \\
        &= \sum_{m_1 \in \mathbb N^{1,1}_l}\sum_{n_1 \in \mathbb N^{1,2}_l} \sum_{m_2 \in \mathbb N^{2,1}_l} \sum_{n_2 \in \mathbb N^{2,2}_l} T(f_1\mathbbm{1}_{B^{m_1}_{j_1(m_1,l)} \times C^{n_1}_{i_1(n_1,l)}}, f_2\mathbbm{1}_{B^{m_2}_{j_2(m_2,l)} \times C^{n_2}_{i_2(n_2,l)}}),
    \end{align*}
    where the series converges in $L^{q,s}(X)$ norm. Therefore, for almost every $x \in X$,
    \begin{align}
        & \Bigg(\sum_{l=1}^{L-1}|T(f_1\mathbbm{1}_{\mathcal{Y}^{1,1}_{N^{1,1}_{l}} \times (\mathcal{Y}^{1,2}_{N^{1,2}_{l+1}} \setminus \mathcal{Y}^{1,2}_{N^{1,2}_{l}})},f_2\mathbbm{1}_{\mathcal{Y}^{2,1}_{N^{2,1}_{l}} \times (\mathcal{Y}^{2,2}_{N^{2,2}_{l+1}} \setminus \mathcal{Y}^{2,2}_{N^{2,2}_{l}})})(x)|^{\rho}\Bigg)^{1/{\rho}} \nonumber\\
        &\leq \Bigg(\sum_{l=1}^{L-1}\Big( \sum_{m_1 \in \mathbb N^{1,1}_l}\sum_{n_1 \in \mathbb N^{1,2}_l} \sum_{m_2 \in \mathbb N^{2,1}_l} \sum_{n_2 \in \mathbb N^{2,2}_l} |T(f_1\mathbbm{1}_{B^{m_1}_{j_1(m_1,l)} \times C^{n_1}_{i_1(n_1,l)}}, f_2\mathbbm{1}_{B^{m_2}_{j_2(m_2,l)} \times C^{n_2}_{i_2(n_2,l)}})(x)| \Big)^{\rho}\Bigg)^{1/{\rho}} \nonumber\\
        &\leq \sum_{m_1=1}^{\infty}\sum_{n_1=1}^{\infty}\sum_{m_2=1}^{\infty}\sum_{n_2=1}^{\infty} \Bigg(\sum_{\substack{l:m_1 \in \mathbb N^{1,1}_l\\n_1 \in \mathbb N^{1,2}_l \\ m_2 \in \mathbb N^{2,1}_l\\ n_2 \in \mathbb N^{2,2}_l}}| T(f_1\mathbbm{1}_{B^{m_1}_{j_1(m_1,l)} \times C^{n_1}_{i_1(n_1,l)}}, f_2\mathbbm{1}_{B^{m_2}_{j_2(m_2,l)} \times C^{n_2}_{i_2(n_2,l)}})(x)|^{\rho}\Bigg)^{1/{\rho}}. \label{finalest1}
    \end{align}
    For fixed $(m_1, n_1) \in \mathbb N^2$ and $(m_2, n_2) \in  \mathbb N^2$, consider the set $A = \{l : m_1 \in \mathbb N^{1,1}_l , n_1 \in N^{1,2}_l ,m_2 \in \mathbb N^{2,1}_l, n_2 \in \mathbb N^{2,2}_l \}$. As before, for every $l,l' \in A$ such that $l \neq l'$, we have $i_1(n_1,l) \neq i_1(n_1, l')$. Therefore, we can dominate the resulting bound in (\ref{finalest1}) by
    \begin{equation}\label{finalest2}
      \sum_{m_1=1}^{\infty}\sum_{n_1=1}^{\infty}\sum_{m_2=1}^{\infty}\sum_{n_2=1}^{\infty} \Bigg(\sum_{j_1=1}^{2^{m_1}} \sum_{i_1=1}^{2^{n_1}} \sum_{j_2=1}^{2^{m_2}} \sum_{i_2=1}^{2^{n_2}} |T(f_1\mathbbm{1}_{B^{m_1}_{j_1} \times C^{n_1}_{i_1}}, f_2 \mathbbm{1}_{B^{m_2}_{j_2} \times C^{n_2}_{i_2}})(x)|^{\rho}\Bigg)^{1/{\rho}}.  
    \end{equation}
    As (\ref{finalest2}) is independent of the increasing sequence we initially considered, we can take supremum over all finite length increasing sequences to obtain the pointwise bound
    $$
    {\widetilde{V}^{\rho}}(T_{\underline{N}^1,\underline{N}^2}(f_1,f_2)(x): \underline{N}^1, \underline{N}^2 \in \mathbb Z^2) \leq  \sum_{m_1=1}^{\infty}\sum_{n_1=1}^{\infty}\sum_{m_2=1}^{\infty}\sum_{n_2=1}^{\infty}S_{m_1,n_1,m_2,n_2}(f_1,f_2)(x),
    $$
    where
    $$S_{m_1,n_1,m_2,n_2}(f_1,f_2)(x)=\Bigg(\sum_{j_1=1}^{2^{m_1}} \sum_{i_1=1}^{2^{n_1}} \sum_{j_2=1}^{2^{m_2}} \sum_{i_2=1}^{2^{n_2}} |T(f_1\mathbbm{1}_{B^{m_1}_{j_1} \times C^{n_1}_{i_1}}, f_2 \mathbbm{1}_{B^{m_2}_{j_2} \times C^{n_2}_{i_2}})(x)|^{\rho}\Bigg)^{1/{\rho}}.$$
    Take $\rho < \min\{q, s\}$. Using the scaling property of Lorentz quasi-norm and its equivalence to a genuine norm, as in Theorem \ref{multi}, with the boundedness of $T$, we obtain
\begin{align*}
    \|S_{m_1,n_1,m_2,n_2}(f_1,f_2)\|_{L^{q,s}(X)} 
    &\leq K^{1/{\rho}}\Big(\sum_{j_1=1}^{2^{m_1}}\sum_{i_1=1}^{2^{n_1}}\sum_{j_2=1}^{2^{n_1}} \sum_{i_2=1}^{2^{n_2}} \|T(f_1\mathbbm{1}_{B^{m_1}_{j_1} \times C^{n_1}_{i_1}}, f_2 \mathbbm{1}_{B^{m_2}_{j_2} \times C^{n_2}_{i_2}})\|^{\rho}_{L^{q,s}(X)} \Big)^{1/{\rho}}\\
    &\leq K^{1/{\rho}}\|T\|_{\prod_{k=1}^{2}L^{p_k,r_k} \to L^{q,s}} \prod_{k=1}^{2}\big(\sum_{j_k=1}^{2^{m_k}} \sum_{i_k=1}^{2^{n_k}} \|f_k \mathbbm{1}_{B^{m_k}_{j_k} \times C^{n_k}_{i_k}}\|^{\rho}_{L^{p_k,r_k}(Y^k)}\big)^{1/{\rho}}.
\end{align*}
    Finally, applying the arguments of Theorem \ref{multi} to $f_1$ and $f_2$ and appropriately summing over $m_1,n_1,m_2$ and $n_2$, we get 
    $$ \|{V}^{\rho}(T_{\underline{N}^1, \underline{N}^k}(f_1,f_2):\underline{N}^1, \underline{N}^2 \in \mathbb Z^2)\|_{L^{q,s}(X)}  \leq A_{\rho,r_1} A_{\rho,r_2} \|T\|_{\prod_{k=1}^{2}L^{p_k,r_k} \to L^{q,s}},$$
    for $r_1<\rho <\min\{q, s\}$ and $r_2 < \rho < \min\{q,s\}$. Using the monotonicity of the $\rho$-variation semi-norm, above inequality is true for all $\rho> r_1$ and $\rho > r_2$. This completes the proof of the theorem.
\end{proof}

\section{The endpoint $\rho=p$}
In this section we discuss the sharpness of the range of $\rho$ in Theorem \ref{multi}. For the Fourier transform operator $\mathcal{F}$, it turns out that for $p>1$ the range is sharp. Suppose that $\mathcal{F}_{N}f=\mathcal{F}(f\mathbbm{1}_{[-N,N]})$, $N>0$ is the Fourier transform of a truncated function and let $g=\chi_{[-1,1]}$ be the characteristic function of the interval $[-1,1]$. We observe that
$$
\mathcal{F}_{N}g(x)=\begin{cases}
                    \frac{\sin (2\pi x)}{\pi x},   & \text{ if $N \ge 1$}\\
                    \frac{\sin (2\pi N x)}{\pi x},  & \text{if $N <1$}.
                                              \end{cases}
$$

We show that ${V}^{p}(\mathcal{F}_{N}g:0<N<1) \not\in L^{p'}$ for $p>1$, where $p'$ is the conjugate exponent of $p$. For $x \in \mathbb R$ consider the sequence $N_{1,x}<\cdots< N_{l,x} < \cdots<N_{L,x}$ where $N_{l,x}=\frac{\pi/2+l\pi}{2\pi x}$. Then
$$
\Bigl|\frac{\sin (2\pi x N_{l+1})-\sin (2\pi x N_{l})}{x}\Bigr|=\frac{2}{|x|},
$$
and for large $x$, $\#\{N_{l,x}:N_{1,x} <\cdots<N_{L,x}<1\} \ge C|x|$. Thus ${V}^{p}(\mathcal{F}_{N}f:0<N<1)(x) \ge \frac{C}{|x|^{1/p'}}$ from which it follows that $\|{V}^{p}(\mathcal{F}_{N}g:N>0)\|_{L^{p'}(\mathbb R)}=\infty$. For $p=1$ the situation is different. In general, let $T:L^1(Y)\to L^\infty(X)$ be a bounded linear operator. Then at the endpoint $p=r=\rho=1$, $q=s=\infty$, Theorem~\ref{multi} holds trivially with $A_{1,1}=1$; that is, $$\|V^1(T_{N}f:N \in \mathbb Z)\|_{L^{\infty}(X)} \leq \|T\|_{L^1 \to L^{\infty}} \|f\|_{L^1(Y)}.$$ Indeed, fix $N_1 <N_2 < \ldots<N_L$. Since $\mathcal Y_{N_l}\subseteq \mathcal Y_{N_{l+1}}$,
linearity and boundedness of $T$ gives
\begin{align*}
    \bigl\|\sum_{l=1}^{L-1}|T(f \mathbbm{1}_{\mathcal{Y}_{N_{l+1}}})-T(f \mathbbm{1}_{\mathcal{Y}_{N_{l}}})|\bigr\|_{L^{\infty}(X)} &= \bigl\|\sum_{l=1}^{L-1} |T(f \mathbbm{1}_{\mathcal{Y}_{N_{l+1}} \setminus \mathcal{Y}_{N_{l}}})|\bigr\|_{L^{\infty}(X)} \\
    &\leq \|T\|_{L^1 \to L^{\infty}}\sum_{l=1}^{L-1} \|f \mathbbm{1}_{\mathcal{Y}_{N_{l+1}} \setminus \mathcal{Y}_{N_{l}}}\|_{L^1(Y)}\\
    &\leq \|T\|_{L^1 \to L^{\infty}} \|f\|_{L^1(Y)},
\end{align*}
where the last inequality holds because the sets $\mathcal Y_{N_{l+1}}\setminus\mathcal Y_{N_l}$ are pairwise disjoint. Hence, for each such sequence, the bound $\sum_{l=1}^{L-1}|T_{N_{l+1}}f(x)-T_{N_l}f(x)|\le\|T\|_{L^1 \to L^{\infty}}\|f\|_{L^1(Y)}$ holds outside a null set. Since there are only countably many finite increasing sequences in $\mathbb Z$, the union of these null sets is null. For $x$ outside it, we may take the supremum over all sequences, obtaining $V^1(T_Nf:N\in\mathbb Z)(x)\le\|T\|_{L^1\to L^\infty}\|f\|_{L^1(Y)}$, which gives the claim. The same argument works verbatim in the multiparameter setting.

In view of the failure of the $L^{p'}$ bound for $p>1$ and the validity of the case $p=1$,
we are led to the following question: given $p>1$, does there exist a constant $C_p$ such that
$$
\|{V}^p(\mathcal{F}_{N}f:N>0)\|_{L^{p',\infty}(X)} \leq C_p \|f\|_{L^p(Y)}?
$$
When $f$ is the characteristic function of an interval then the above inequality is true, which gives us an indication of a possibility of a restricted weak type $(p,q)$ result at the endpoint $\rho=p$. One can also ask the same question in the general Christ--Kiselev setup.

\end{document}